\documentclass{article}
\usepackage[utf8]{inputenc}
\usepackage{amsmath}
\usepackage{amssymb}
\usepackage{amsthm}
\usepackage{tikz-cd}
\usepackage{graphicx}
\usepackage{mathtools}
\usepackage{extarrows}
\usepackage{booktabs}
\usepackage{pythonhighlight}
\usepackage{blkarray}
\usepackage{url}

\usepackage{biblatex}
\newtheorem{lemma}{Lemma}
\newtheorem{theorem}{Theorem}

\usepackage{geometry}
\newcommand{\stupidsum}{\sum_{\substack{i^\text{th},j^\text{th}\text{ vertex}\\\text{adjacent, } i<j}}}

\title{Realizing abstract simplicial complexes with specified edge lengths}
\author{Matthew Ellison\\matthew.ellison.gr@dartmouth.edu\\
Dartmouth College}
\date{}
\begin{document}

\maketitle

\begin{abstract}
    For finite abstract simplicial complex $\Sigma$, initial realization $\alpha$ in $\mathbb{E}^d$, and desired edge lengths $L$, we give practical sufficient conditions for the existence of a non-self-intersecting perturbation of $\alpha$ realizing the lengths $L$. We provide software to verify these conditions by computer and optionally assist in the creation of an initial realization from abstract simplicial data. Applications include proving the existence of a planar embedding of a graph with specified edge lengths or proving the existence of polyhedra (or higher-dimensional polytopes) with specified edge lengths. 

\end{abstract}

\section{Introduction}
Consider the problem of whether a finite abstract simplicial complex\footnote{Recall an \textit{abstract simplicial complex} is a set $S$ of vertices together with a set of subsets $\Delta$ corresponding to edges, triangular faces, tetrahedral volumes, etc. A simplex must contain all its faces, so $\Delta$ must be closed under subsets.} may be realized in $\mathbb{E}^d$ with flat faces, no self-intersection, and prescribed edge lengths. 

One resolution is to assign coordinates to the vertices so the edge lengths are exactly as prescribed. All faces are then fixed, and one may verify the realization is non-self-intersecting. More generally, any rigorous exact geometric construction resolves the problem.\\

We pursue a different approach. Suppose we have an assignment of coordinates to the polytope vertices so the edge lengths are approximately correct. Such an assignment might be produced by computer simulation, or, perhaps, approximate physical construction and measurement. If one proved the existence of a perturbed assignment where (i) the edge lengths are exactly correct and (ii) the perturbated realization is non-self-intersecting, then the problem is also resolved.\\ 

In this paper, we prove practical sufficient conditions for such a perturbation to exist. We also provide code, in the form of a Python3 package, which may be used to prove existence from abstract simplicial data, desired edge lenghts, and an approximate realization. The code can optionally assist in creating the approximate realization. See Section \ref{examples} for example applications.\\ 

Related results include Steinitz's Theorem (\cite{steinitz_orig}\cite{steinitz_writeup}\cite{steinitz_writeup_2}), which gives necessary and sufficient conditions for a graph to be the net of a three-dimensional convex polyhedron; recent work by Abrahamsen at al (\cite{abrahamsen2021geometric}) which establishes the NP-hardness of deciding whether a $k$-dimensional abstract simplicial complex admits a geometric embedding in $\mathbb{R}^d$ for $d\geq 3$ and $k=d-1, d$; and work by Cabello, Demaine, and Rote (\cite{demaine}) on the planar embedding of graphs with specified edge lengths.\\
\section{Perturbing to obtain edge lengths}
In this section we provide sufficient conditions for the existence of a perturbed realization --- possibly self-intersecting --- with edge lengths exactly correct. We will address self-intersection in the following section.\\

Let $\Sigma$ be an finite abstract simplicial complex with vertex set $V$ and edge set $E$. Higher dimensional structure is not relevant in this section. Consider the map $l^2: \mathbb{R}^{d|V|}\rightarrow \mathbb{R}^{|E|}$, which takes coordinates of vertices to square lengths of edges.\footnote{Fixing, once and for all, an ordering of the vertex components and edges.} We choose square lengths of edges because we will be taking derivatives shortly. 

Let $\alpha\in\mathbb{R}^{d|V|}$ denote our initial approximate realization, and let $l^2_* \in \mathbb{R}^{|E|}$ denote our desired square edge lengths for $\Sigma$.\\

In this language, we want to find a point $x$ near $\alpha$ such that $l^2(x) = l^2_*$. Our approach will be to find sufficient conditions for a ball around $\alpha$ to surject, under $l^2$, onto a neighborhood of $l^2(\alpha)$ containing $l^2_*$.\\

We begin by analyzing the derivatives of $l^2$.
\begin{lemma}\label{l2_jac_lemma}Let $x,y$ be two points in $\mathbb{R}^d$. Then
$$\frac{\partial}{\partial x} ||x-y||^2 = 2(x-y).$$
\end{lemma}
\begin{proof}
\begin{align*}
    \frac{\partial}{\partial x} ||x-y||^2= \frac{\partial}{\partial x}[(x-y)^T(x-y)]=\frac{\partial}{\partial x}x^Tx - \frac{\partial}{\partial x}x^Ty - \frac{\partial}{\partial x} y^T x=2x - y - y.
\end{align*}
\end{proof}
To unpack this compact notation, let $x_i$ and $y_i$ denote the $i$th components of $x$ and $y$. We then have $\frac{\partial}{\partial x_i} ||x-y||^2 = 2(x_i-y_i)$.\\

Lemma \ref{l2_jac_lemma} makes it straightforward to compute the Jacobian of $l^2$. For the second derivatives, we may apply the following result:
\begin{lemma}\label{l2_hess_lemma}
 Let $x,y$ be two points in $\mathbb{R}^d$, and let $I$ be the $d\times d$ identity matrix. Then
    $$\frac{\partial^2}{\partial x\partial y}||x-y||^2 = \frac{\partial^2}{\partial y\partial x}||x-y||^2 =-2I,\ \frac{\partial^2}{\partial x^2}||x-y||^2=\frac{\partial^2}{\partial y^2}||x-y||^2= 2I.$$
\end{lemma}
\begin{proof}[]
\end{proof}
An example unpacking is $\frac{\partial^2}{\partial x_i\partial y_j}||x-y||^2 =(-2I)_{ij}=-2\delta_{ij}$.\\

The higher derivatives of $l^2$ vanish. Consider the Taylor expansion of $l^2$ about $\alpha$. For a perturbation $\epsilon$, which perturbs the $i$th vertex by $\epsilon_i$, we have\\
\begin{equation}\label{taylor_expansion}
l^2(\alpha + \epsilon) - l^2(\alpha) = Dl^2(\alpha)\epsilon + \frac{1}{2} \stupidsum \left[ 2\epsilon_i^T \epsilon_i-4\epsilon_i^T \epsilon_j+2\epsilon_j^T \epsilon_j\right]e_{ij}.
\end{equation}
Here $e_{ij}$ denotes the unit vector corresponding to the respective edge.\\

In Equation \ref{taylor_expansion}, we are interested in lower-bounding the left hand side's magnitude. Preparing to apply the reverse triangle inequality, we note the following.
\begin{align*}
\intertext{Let $\sigma_{\min}$ denote the smallest singular value of $Dl^2(\alpha)$.}
||Dl^2(\alpha)\epsilon||&\geq \sigma_{\min} ||\epsilon||\\
\left|\left|\frac{1}{2}\sum_{(i,j)\in \text{Edge}(T)} \left[ 2\epsilon_i^T \epsilon_i-4\epsilon_i^T \epsilon_j+2\epsilon_j^T \epsilon_j\right]e_{(i,j)}\right|\right|&\leq \frac{1}{2}\sqrt{|E|}\left(8||\epsilon||^2\right).
\end{align*}
Thus, for $\epsilon$ sufficiently small that $\sigma_{\min}\geq 4\sqrt{|E|}||\epsilon||$, we have 
$$||f(\alpha+\epsilon)-f(\alpha)||\geq \sigma_{\min}||\epsilon||-4\sqrt{|E|}||\epsilon||^2.$$
With a little topological work, we can now push this toward the kind of result we want.
\begin{lemma}
    Suppose $d|V|\geq |E|$, $\sigma_{\min} >0$ (so we are locally surjective), and fix $\delta$ sufficiently small that $\sigma_{\min}> 4\sqrt{|E|}\ \delta$. Then the image under $l^2$ of the closed ball of radius $\frac{\sigma_{\min}}{4\sqrt{|E|}}$ around $\alpha$ takes every value in the ball of radius $\sigma_{\min}\delta-4\sqrt{|E|}\delta^2$ centered at $l^2(\alpha)$.
\end{lemma}
\begin{proof}
    We will apply a sequence of reductions to $l^2$ until a missing value in the desired ball yields a contradiction. This contradiction will be by the fact that there is no retraction of the ball onto its boundary.\\
    First we perturb away the contribution of the quadratic term on the boundary and shift to the origin. \\
    Let $\eta:[0,\frac{\sigma_{\min}}{4\sqrt{|E|}}]\rightarrow [0,1]$ be a smooth bump-style function which is 1 on $[0,\delta]$ and smoothly decreases to 0 at $\frac{\sigma_{\min}}{4\sqrt{E}}$.
    Now define $f$ by
    $$f(\epsilon) = Dl^2(\alpha)\epsilon + \frac{\eta(||\epsilon||)}{2} \stupidsum \left[ 2\epsilon_i^T \epsilon_i-4\epsilon_i^T \epsilon_j+2\epsilon_j^T \epsilon_j\right]e_{ij}.$$
    By our bounding work above, points in the spherical shell between radius $\delta$ and $\frac{\sigma_{\min}}{4\sqrt{E}}$ cannot take values in the desired ball, either before or after applying $\eta$. The desired statement then reduces to showing that the image under $f$ of the closed ball of radius $\frac{\sigma_{\min}}{4\sqrt{|E|}}$ around 0 takes every value in the ball of radius $\sigma_{\min}\delta-4\sqrt{|E|}\delta^2$ about 0.\\
    Next we restrict $f$ to a same-radius sub-ball of the domain ball of the same dimension as the codomain, in such  a way that the restricted Jacobian is invertible at 0. We can do this because the original Jacobian is full rank by assumption. Call this restricted function $g$, so that
    $$g(\epsilon) = Dg(0)\epsilon + \frac{\eta(||\epsilon||)}{2}\stupidsum \left[ 2\epsilon_i^T \epsilon_i-4\epsilon_i^T \epsilon_j+2\epsilon_j^T \epsilon_j\right]e_{ij}.$$
    It now suffices to establish the (stronger) statement that the image under $g$ of the (now lower-dimensional) closed ball of radius $\frac{\sigma_{\min}}{4\sqrt{|E|}}$ around 0 takes every value in the ball of radius $\sigma_{\min}\delta-4\sqrt{|E|}\delta^2$ about 0.\\
    We now define a final function $h$ which is simply $g$ multiplied on the left by $Dg(0)^{-1}$:
    $$h(\epsilon) = \epsilon + Dg(0)^{-1}\frac{\eta(||\epsilon||)}{2}\stupidsum \left[ 2\epsilon_i^T \epsilon_i-4\epsilon_i^T \epsilon_j+2\epsilon_j^T \epsilon_j\right]e_{ij}.$$
    To finish, it will suffice to show that the image under $h$ of the closed ball of radius $\frac{\sigma_{\min}}{4\sqrt{|E|}}$ about 0 contains each point in this ball's interior. But suppose the image did not contain some point $p$, then, since $h$ fixes the boundary, we may obtain a contradiction in the same way as the typical proof of Brauer's Fixed Point theorem.
\end{proof}
Maximizing the radius of the image ball over $\delta$ we then obtain our section's main result.
\begin{theorem}\label{can perturb}
Let $\Sigma$ be an abstract simplicial complex with vertex set $V$ and edge set $E$, which we wish to embed in $\mathbb{E}^d$. Fix an ordering of the vertex coordinates and edges, and let $l^2: \mathbb{R}^{d|V|}\rightarrow \mathbb{R}^{|E|}$ be the map taking an assignment of vertex coordinates to square edge lengths. Let $\alpha\in \mathbb{R}^{d|V|}$ be an initial embedding and let $l^2_*\in\mathbb{R}^{|E|}$ be the desired square edge lengths. Then the following conditions are sufficient for the existence of a realization with the desired edge lengths --- \textbf{potentially self-intersecting}:
\begin{itemize}
    \item $d|V| \geq |E|$
    \item $\sigma_{\min} > 0$, where $\sigma_{\min}$ is the smallest singular value of $Dl^2(\alpha)$
    \item $\rho < \sigma_{\min}^2 / (16\sqrt{|E|})$, where $\rho =||l^2_* - l^2(\alpha)||$.
\end{itemize}
\end{theorem}
\begin{proof}[]\end{proof}
Recall that Lemma \ref{l2_jac_lemma} may be used to easily compute $Dl^2(\alpha)$. We conclude this section with an observation which will be useful later. 
\begin{lemma}\label{sol dist}
Suppose the conditions of Theorem \ref{can perturb} are satisfied. Then we have a realization $\alpha^*\in \mathbb{R}^{d|E|}$ with the desired edge lengths such that
$$||\alpha^* - \alpha|| \leq \frac{\sigma_{\min}-\sqrt{\sigma_{\min}^2-16\rho\sqrt{|E|}}}{8\sqrt{|E|}}.$$
\end{lemma}
\begin{proof}
    By Lemma 3, we need only move a distance $\delta$ such that 
    $$
        \sigma_{\min} > 4\sqrt{|E|}\delta \text{ and }
        \rho < \sigma_{\min}\delta - 4\sqrt{|E|}\delta^2.$$
    Focusing on the second condition, this will be satisfied between the roots of 
    $4\sqrt{|E|}\delta^2-\sigma_{\min}\delta +\rho = 0$, which occur at 
    $$\delta = \frac{\sigma_{\min} \pm \sqrt{\sigma_{\min}^2-16\rho\sqrt{|E|}}}{8\sqrt{|E|}}.$$
    Note that the discriminant is positive by the third condition of Theorem \ref{can perturb}, and that both roots are positive. Note also that the $\sigma_{\min} > 4\sqrt{|E|}$ condition is satisfied throughout the whole open interval between the roots. Thus we can take our $\delta$ to be any value in the open interval between the roots, and taking the infimum gives our result.
\end{proof}

\section{Enforcing non-self-intersection}\label{sec-self-intersection}
We now extend Theorem \ref{can perturb} to give sufficient conditions for the existence of a \textit{non-self-intersecting} realization with the desired edge lengths. This is not so difficult, although implementing the computations takes some work (see Appendix \ref{convex distance}).\\

First we formalize self-intersection. Let $\Sigma$ be an abstract simplicial complex. Two simplices $\sigma_1$ and $\sigma_2$ of $\Sigma$ are \textit{non-adjacent} if their vertex sets are disjoint. For a realization $\alpha$ of $\Sigma$ in $\mathbb{E}^d$, let $d_\alpha(\sigma_1, \sigma_2)$ be the minimum distance between a point from each of the realizations of $\sigma_1$ and $\sigma_2$.\\
Define the \textit{collision distance} of realization $\alpha$ to be 
$$\text{CD}_\alpha(\Sigma) = \min_{\substack{\sigma_1,\sigma_2\in\Sigma\\\text{non-adjacent}}}d_\alpha(\sigma_1, \sigma_2).$$
A realization $\alpha$ of $\Sigma$ is self-intersecting exactly when $\text{CD}_\alpha(\Sigma) = 0$.\footnote{It might be more pleasing, though not needed, to define the collision distance for a self-intersecting realization $\alpha$ to be minus the minimum distance, in some sense, to a non-self-intersecting realization $\alpha^*$.}
\begin{lemma}
Suppose vertex $v$ is moved, in realization $\alpha$, to $v^*$, yielding realization $\alpha^*$. Then $$\text{CD}_{\alpha^*}(\Sigma) \geq \text{CD}_\alpha(\Sigma) - ||v^* - v||.$$
Briefly, moving a vertex by $\delta$ can at most decrease collision distance by $\delta$.
\end{lemma}
\begin{proof}
    It suffices to show, for any $\sigma_1, \sigma_2$ non-adjcent in $\Sigma$, that $$d_{\alpha^*}(\sigma_1, \sigma_2) \geq d_{\alpha}(\sigma_1, \sigma_2) - ||v^* - v||.$$
    Let $m,n$ be the respective dimensions of $\sigma_1$, $\sigma_2$. The case of interest is when exactly one contains $v$ (if neither do the claim is clear and it can't be both since they are non-adjacent). Without loss of generality, suppose $\sigma_1$ contains $v$ with initial vertex coordinates $v, v_2, ..., v_{m+1}$ and let $\sigma_2$ have coordinates $w_1, ..., w_{n+1}$.\\
    
    Let $[\alpha_1 : ... : \alpha_{m+1}]$ and $[\beta_1: ... : \beta_{n+1}]$ be the barycentric coordinates of a closest pair in $\alpha^*$. Expanding out, we then have
    \begin{align*}
        d_{\alpha^*}(\sigma_1, \sigma_2) + \alpha_1||v^* - v||&= \left|\left|\alpha_1 v_1^* + \sum_{i > 1} \alpha_i v_i -\sum_i\beta_i w_i\right|\right| + ||\alpha_1(v_1 - v_1^*)|| \\
        & \geq  
        \left|\left|\alpha_1 v_1 + \sum_{i > 1} \alpha_i v_i -\sum_i\beta_i w_i\right|\right| \\
        &\geq d_\alpha(\sigma_1, \sigma_2).
    \end{align*}
    Thus 
    \begin{align*}
        d_{\alpha^*}(\sigma_1, \sigma_2) &\geq d_\alpha(\sigma_1, \sigma_2) -\alpha_1||v^* - v||\\
        &\geq d_\alpha(\sigma_1, \sigma_2) - ||v^* - v||.
    \end{align*}
\end{proof}
\begin{lemma}
Let $\Sigma$ be an abstract simplicial complex with vertex set $V$ and edge set $E$. Fixing an ordering of vertex coordinates, let $\alpha\in\mathbb{R}^{d|V|}$ be a non-self-intersecting realization in $\mathbb{E}^d$. Let $\alpha^*\in\mathbb{R}^{d|V|}$ be another realization. If
$$||\alpha^* - \alpha|| < \frac{1}{\sqrt{|V|}}\text{CD}_\alpha(\Sigma),$$
then $\alpha^*$ is also non-self-intersecting.
\end{lemma}
\begin{proof}
    Let $\epsilon = ||\alpha^* - \alpha||$, and let $\epsilon_i$ be the displacement of the $i$th vertex from $\alpha$ to $\alpha^*$.
    If we consider moving the vertices one by one, we see 
    \begin{align*}
        \text{CD}_{\alpha^*}(\Sigma) &\geq \text{CD}_{\alpha}(\Sigma) - \sum ||\epsilon_i||,\\
        \intertext{and so we are safe if}
        \text{CD}_{\alpha}(\Sigma) &> \sum ||\epsilon_i||.\\
    \end{align*}
    We then obtain our result by weakening via following inequality:
    $$\sum ||\epsilon_i|| \leq \sqrt{V}||\epsilon||.$$
\end{proof}
Combining this result with Theorem \ref{can perturb} and Lemma \ref{sol dist}, we obtain our desired existence test.
\begin{theorem}\label{main}
    Let $\Sigma$ be an abstract simplicial complex with vertex set $V$ and edge set $E$, which we wish to embed in $\mathbb{E}^d$. Fix an ordering of the vertex coordinates and edges, and let $l^2: \mathbb{R}^{d|V|}\rightarrow \mathbb{R}^{|E|}$ be the map taking an assignment of vertex coordinates to square edge lengths. Let $\alpha\in \mathbb{R}^{d|V|}$ be an initial non-self-intersecting realization and let $l^2_*\in\mathbb{R}^{|E|}$ be the desired square edge lengths. Then the following conditions are sufficient for the existence of a realization with the desired edge lengths and no self-intersections:
\begin{itemize}
    \item $d|V| \geq |E|$
    \item $\sigma_{\min} > 0$, where $\sigma_{\min}$ is the smallest singular value of $Dl^2(\alpha)$
    \item $\rho < \sigma_{\min}^2 / (16\sqrt{|E|})$, where $\rho =||l^2_* - l^2(\alpha)||$
    \item $\frac{\sigma_{\min}-\sqrt{\sigma_{\min}^2-16\rho\sqrt{|E|}}}{8\sqrt{|E|}} < \frac{1}{\sqrt{|V|}}\text{CD}_\alpha(\Sigma).$
\end{itemize}
\end{theorem}
\section{Proving existence by computer}
Suppose we have an abstract simplicial complex $\Sigma$, and we wish to prove the existence of a non-self-intersecting realization in $\mathbb{E}^d$ with desired square edge lengths $l_*^2$. For example, we might wish to prove the existence of the regular icosahedron in the usual three space. It is tempting to apply Theorem \ref{main} by computer in the following way:
\begin{enumerate}
    \item Obtain a realization $\alpha$ of $\Sigma$ in $\mathbb{E}^d$ which is non-intersecting (perhaps not rigorously proven) and with edge lengths reasonably correct.
    \item Prove $\alpha$ is non-self-intersecting.
    \item Prove $\alpha$ satisfies the 4 inequalities of Theorem \ref{main}.
\end{enumerate}
This is the approach followed by our code. In this section we describe in more detail how the steps may be rigorously carried out by computer. Any one of the verification steps may fail along the way, in which case the proof is aborted (or one modifies the described process to salvage it), but for cleaner exposition we describe the process as if each stage will be successful. \\

We will suppose the square lengths in $l_*^2$ are rational for this section. This is what our software supports, but this discussion (and the software) would generalize to any countable ordered subfield of $\mathbb{R}$ with a little work.

\subsection{Initial realization}
The coordinates of an approximate realization may be directly input by the user. Alternatively, we have found it practical to generate coordinates by a physics-inspired simulation:
\begin{enumerate}
    \item Randomly initialize the vertex coordinates.
    \item Iterate time steps where the vertices move according to repulsive forces between all vertices and spring forces between vertices joined by an edge (with spring length desired length).
    \item Iterate more time steps with just the spring forces.
    \item Start again if the resulting realization is heuristically self-intersecting.
    \item Approximate the coordinates by fractions so we have an exact representation the computer can work with for the remaining steps.
\end{enumerate}

\subsection*{4.2 Proving non-self-intersection}
At this point we have an approximate realization $\alpha$ with rational coordinates.\\

To verify it is non-self-intersecting, it suffices to check that each pair of non-adjacent simplices in $\Sigma$ is non-self-intersecting.\footnote{And we can save time by only checking that maximal pairs of non-adjacent simplices are non-self-intersecting.} This in turn may be accomplished by computing $d_\alpha(\sigma_1, \sigma_2)$ for each such pair as described in Appendix \ref{convex distance}.

\subsection{Checking the inequalities}\label{checking inequalities}
The final step is to attempt to prove it satisfies the 4 inequalities of Theorem \ref{main}.\\

It is clear how to test the first inequality, $d|V| \geq |E|$.\\

For the second inequality, $\sigma_{\min} > 0$, we may apply the process of Appendix \ref{bound sigma} to obtain a good rational interval $[\sigma_l, \sigma_u]$ containing $\sigma_{\min}$ and then check $\sigma_l > 0$.\\

For the third inequality, 
$$\rho < \sigma_{\min}^2 / (16\sqrt{|E|}),\ \text{where } \rho =||l^2_* - l^2(\alpha)||,$$
we can first compute rational intervals around $\rho$ and $\sqrt{|E|}$ using the process of Appendix \ref{bound sqrt}, and then, together with our interval $[\sigma_l, \sigma_u]$ from before, verify the inequality with interval arithmetic.\\

Finally, we have the fourth inequality:
$$\frac{\sigma_{\min}-\sqrt{\sigma_{\min}^2-16\rho\sqrt{|E|}}}{8\sqrt{|E|}} < \frac{1}{\sqrt{|V|}}\text{CD}_\alpha(\Sigma).$$
To verify this, we can first obtain rational intervals around $\sqrt{|V|}$ and $\text{CD}_{\alpha}(\Sigma)$. Note that we can find $\left(\text{CD}_{\alpha}(\Sigma)\right)^2$ exactly using the process of Appendix \ref{convex distance}. Then, together with our intervals around $\sigma_{\min}$, $\rho$, and $\sqrt{|E|}$ from before, we can do interval arithmetic on the left and right hand sides to verify the inequality.

\section{Example proofs using the supplemental code}\label{examples}
The code is available as a Python 3 package called \texttt{shape-existence}. To use it, you can first install Python 3 and then run the package install command from your terminal: $$\texttt{pip install shape-existence}$$
The following examples give code which would be entered into a file saved with extension \texttt{.py}. To run the \texttt{.py} file, one could navigate to its folder in terminal and run $\texttt{python3 [filename]}$, or, likely, just double click on the file. 
\subsection{30-60-90 triangle}
We prove the existence of a $\pi/6-\pi/3-\pi/2$ triangle with side lengths $1$, $\frac{1}{2}$, and $\frac{\sqrt{3}}{2}$.\\
\par\noindent\rule{\textwidth}{0.4pt}
\tiny
\begin{verbatim}
from shape_existence.complexes_and_proofs import AbstractSimplicialComplex, Fraction
ASC = AbstractSimplicialComplex

triangle = ASC(mode = "maximal_simplices", data = [["a", "b"], ["b", "c"], ["c", "a"]])
square_sides= {("a", "b") : 1, ("b", "c") : Fraction(1,4), ("c", "a") : Fraction(3,4)}
triangle_realized = triangle.heuristic_embed(dim = 2, desired_sq_lengths = square_sides, final_round_digits = 8)
triangle_realized.save_as_obj("triangle_30_60_90.obj", "./obj_files/")
triangle_realized.prove_existence(desired_sq_lengths = square_sides, verbose = True)
\end{verbatim}
\par\noindent\rule{\textwidth}{0.4pt}
\vspace{.1cm}
\normalsize
Running this code yields the following text output (because \texttt{verbose} was set to \texttt{True} in the call to \texttt{prove\_existence}):\\
\par\noindent\rule{\textwidth}{0.4pt}
\tiny
\begin{verbatim}
Attempting to prove existence

Starting realization:
	Abstract data:	
		mode: maximal_simplices
		data: [['a', 'b'], ['b', 'c'], ['c', 'a']]
	Coordinate Data:
		b : [3914567 / 6250000, 63520223 / 100000000]
		a : [27779707 / 50000000, -226433 / 625000]
		c : [104057459 / 100000000, 35519863 / 100000000]

Desired square lengths:
	('a', 'b') : 1
	('b', 'c') : 1 / 4
	('c', 'a') : 3 / 4

Checking inequality 1:
	 d  = 2
	|V| = 3
	|E| = 3
	Success: d|V| >= |E|

Checking self-intersection:
	Square collision distance = 18749999713556450281734401664681 / 99999999862479730000000000000000
	Collision distance in [43301269 / 100000000, 4330127 / 10000000] ~ [0.43301, 0.43301]
	Success: starting realization non-self-intersecting

Checking inequality 2:
	sigma_min in [2651 / 2000, 13257 / 10000] ~ [1.3255, 1.3257]
	Success: sigma_min > 0

Checking inequality 3:
	rho_squared = 6139541520423783 / 50000000000000000000000000000000
	rho in [6925689 / 625000000000000, 13175689 / 625000000000000] ~ [0.0, 0.0]
	sigma_min ^ 2 / (16 * E ^ .5) in [175695025 / 2771281296, 175748049 / 2771281280] ~ [0.0634, 0.06342]
	Success: rho < sigma_min ^ 2 / (16 * E ^ .5)

Checking inequality 4:
	LHS NUM := sigma_min - [sigma_min ^ 2 - 16 * rho * |E| ^ .5 ] ^ .5 in [-19989 / 100000000, 20023 / 100000000] ~ [-0.0002, 0.0002]
	LHS DEN := 8 * |E| ^ .5 in [4330127 / 312500, 173205081 / 12500000] ~ [13.85641, 13.85641]
	LHS     := (LHS NUM) / (LHS DEN) in [-19989 / 1385640640, 20023 / 1385640640] ~ [-1e-05, 1e-05]
	CD / |V| ^ .5 in [43301269 / 173205081, 1 / 4] ~ [0.25, 0.25]
	Success: LHS < CD / |V| ^ .5

Success: existence proven
(187489 / 999956, ([0, 0], [54125 / 249989, 187489 / 499978]))
\end{verbatim}
\par\noindent\rule{\textwidth}{0.4pt}
\normalsize
\vspace{.1cm}
The above proof log documents a successful proof of existence using Theorem \ref{main}.\\
Notes on the code:
\begin{itemize}
    \item We began by creating our triangle as an abstract simplicial complex using the \texttt{AbstractSimplicialComplex} class (renamed \texttt{ASC}). We used mode ``maximal\_simplices" to specify the structure, where the maximal simplices are 1-simplices, the three sides of the triangle. 
    \item We then specified our desired square side lengths, and used them to create a heuristic embedding of the triangle in two dimensions using the \texttt{heuristic\_embed} function (under the hood this is doing a physics-inspired simulation like described above). We specified the square edge lengths using the provided package-provided Fraction class.
    \item The heuristic embedding procedure works in floating point numbers, and then at the end converts to Fractions to yield an exact realization (which hopefully approximates the desired edge lengths). The \texttt{final\_round\_digits} parameter specifies how many floating point digits are preserved in this final conversion.
    \item We used the package-provided \texttt{save\_as\_obj} function, which takes in the type \texttt{RealizedSimplicialComplex} (what \texttt{heuristic\_embed} outputs) and saves a 3D model in the \texttt{obj} format.\footnote{These files can, in 2023 at least, be viewed and rotated on Macs with just the space-bar `preview'.} These files might help convince you that the heuristic embedding is reasonable. We previously created a folder called \texttt{obj\_files} in the directory where the code was run.
\end{itemize}
Notes on the proof log, which is hopefully self-explanatory:
\begin{itemize}
    \item The proof log begins by describing the abstract simplicial complex, starting approximate realization, and desired square edge lengths.
    \item The computed collision distance for the approximate realization is very close to $\sqrt{3}/4$, the shortest altitude of the desired triangle. 
\end{itemize}
\subsection{Icosahedron}
We can follow the same process to prove the existence of the icosahedron.
\par\noindent\rule{\textwidth}{0.4pt}
\tiny
\begin{verbatim}
from shape_existence.complexes_and_proofs import AbstractSimplicialComplex, Fraction
ASC = AbstractSimplicialComplex

icosahedron = ASC(mode = "maximal_simplices",
                data = [["t", "a1", "a2"], ["t", "a2", "a3"], ["t", "a3", "a4"], ["t", "a4", "a5"], ["t", "a5", "a1"],
                        ["a1", "a2", "b1"], ["a2", "a3", "b2"], ["a3", "a4", "b3"], ["a4", "a5", "b4"], ["a5", "a1", "b5"],
                        ["b1", "b2", "a2"], ["b2", "b3", "a3"], ["b3", "b4", "a4"], ["b4", "b5", "a5"], ["b5", "b1", "a1"],
                        ["b", "b1", "b2"], ["b", "b2", "b3"], ["b", "b3", "b4"], ["b", "b4", "b5"], ["b", "b5", "b1"]])
icosahedron_realized = icosahedron.heuristic_embed(dim = 3, desired_sq_lengths = {"default" : 1}, final_round_digits = 9)
icosahedron_realized.save_as_obj("icosahedron.obj", "./obj_files/")
icosahedron_realized.prove_existence(desired_sq_lengths = {"default" : Fraction(1)}, verbose = True)
\end{verbatim}
\par\noindent\rule{\textwidth}{0.4pt}
\vspace{.1cm}
\normalsize
And here is the resulting successful proof log:\\
\par\noindent\rule{\textwidth}{0.4pt}
\tiny
\begin{verbatim}
Attempting to prove existence

Starting realization:
	Abstract data:	
		mode: maximal_simplices
		data: [['t', 'a1', 'a2'], ['t', 'a2', 'a3'], ['t', 'a3', 'a4'], ['t', 'a4', 'a5'], ['t', 'a5', 'a1'], ['a1', 'a2', 'b1'],
    ['a2', 'a3', 'b2'], ['a3', 'a4', 'b3'], ['a4', 'a5', 'b4'], ['a5', 'a1', 'b5'], ['b1', 'b2', 'a2'], ['b2', 'b3', 'a3'], 
    ['b3', 'b4', 'a4'], ['b4', 'b5', 'a5'], ['b5', 'b1', 'a1'], ['b', 'b1', 'b2'], ['b', 'b2', 'b3'], ['b', 'b3', 'b4'], 
    ['b', 'b4', 'b5'], ['b', 'b5', 'b1']]
	Coordinate Data:
		b3 : [78001449 / 250000000, 1442907833 / 1000000000, -9027 / 10000000]
		a3 : [-43340621 / 200000000, 120056377 / 200000000, -103120161 / 1000000000]
		b4 : [693725959 / 1000000000, 1458019437 / 1000000000, 9232517 / 10000000]
		a1 : [794405291 / 1000000000, -156807867 / 1000000000, 908071661 / 1000000000]
		a4 : [-223966263 / 1000000000, 43705607 / 40000000, 383621077 / 500000000]
		a2 : [412685129 / 1000000000, -10744967 / 62500000, -16082739 / 1000000000]
		b2 : [176369499 / 250000000, 661366989 / 1000000000, -485024109 / 1000000000]
		t : [-80871507 / 500000000, 47311007 / 500000000, 378930187 / 500000000]
		b1 : [166297169 / 125000000, 193459789 / 1000000000, 69963403 / 500000000]
		a5 : [400933091 / 1000000000, 24989319 / 40000000, 139219307 / 100000000]
		b : [634077051 / 500000000, 23829559 / 20000000, 74654293 / 500000000]
		b5 : [1323114193 / 1000000000, 685818079 / 1000000000, 505144561 / 500000000]

Desired square lengths:
	default : 1

Checking inequality 1:
	 d  = 3
	|V| = 12
	|E| = 30
	Success: d|V| >= |E|

Checking self-intersection:
	Square collision distance = 62499999704038118952979912461140044476898687719549649 / 86372875554807236480463023424935793500000000000000000
	Collision distance in [2126627 / 2500000, 85065081 / 100000000] ~ [0.85065, 0.85065]
	Success: starting realization non-self-intersecting

Checking inequality 2:
	sigma_min in [7653 / 5000, 3827 / 2500] ~ [1.5306, 1.5308]
	Success: sigma_min > 0

Checking inequality 3:
	rho_squared = 3928473408881149031 / 50000000000000000000000000000000000
	rho in [443197101 / 50000000000000000, 943197101 / 50000000000000000] ~ [0.0, 0.0]
	sigma_min ^ 2 / (16 * E ^ .5) in [6507601 / 243432248, 14645929 / 547722557] ~ [0.02673, 0.02674]
	Success: rho < sigma_min ^ 2 / (16 * E ^ .5)

Checking inequality 4:
	LHS NUM := sigma_min - [sigma_min ^ 2 - 16 * rho * |E| ^ .5 ] ^ .5 in [-799 / 4000000, 4011 / 20000000] ~ [-0.0002, 0.0002]
	LHS DEN := 8 * |E| ^ .5 in [547722557 / 12500000, 273861279 / 6250000] ~ [43.8178, 43.8178]
	LHS     := (LHS NUM) / (LHS DEN) in [-19975 / 4381780456, 20055 / 4381780456] ~ [-0.0, 0.0]
	CD / |V| ^ .5 in [42532540 / 173205081, 85065081 / 346410161] ~ [0.24556, 0.24556]
	Success: LHS < CD / |V| ^ .5

Success: existence proven
\end{verbatim}
\par\noindent\rule{\textwidth}{0.4pt}
\normalsize
\vspace{.1cm}
Code notes:
\begin{itemize}
    \item In specifying our desired square lengths we used the \texttt{"default"} keyword to say that every square edge length not explicitly given (all of them in this case) should have desired square value 1.
\end{itemize}
Proof log notes:
\begin{itemize}
    \item The collision distance for the starting realization $\approx.8506$ is correct, consider disjoint edges on adjacent faces.
\end{itemize}
Here is a screenshot of the 3D model produced by the call to \texttt{save\_to\_obj}:\\
\begin{center}\includegraphics[width = .5\linewidth]{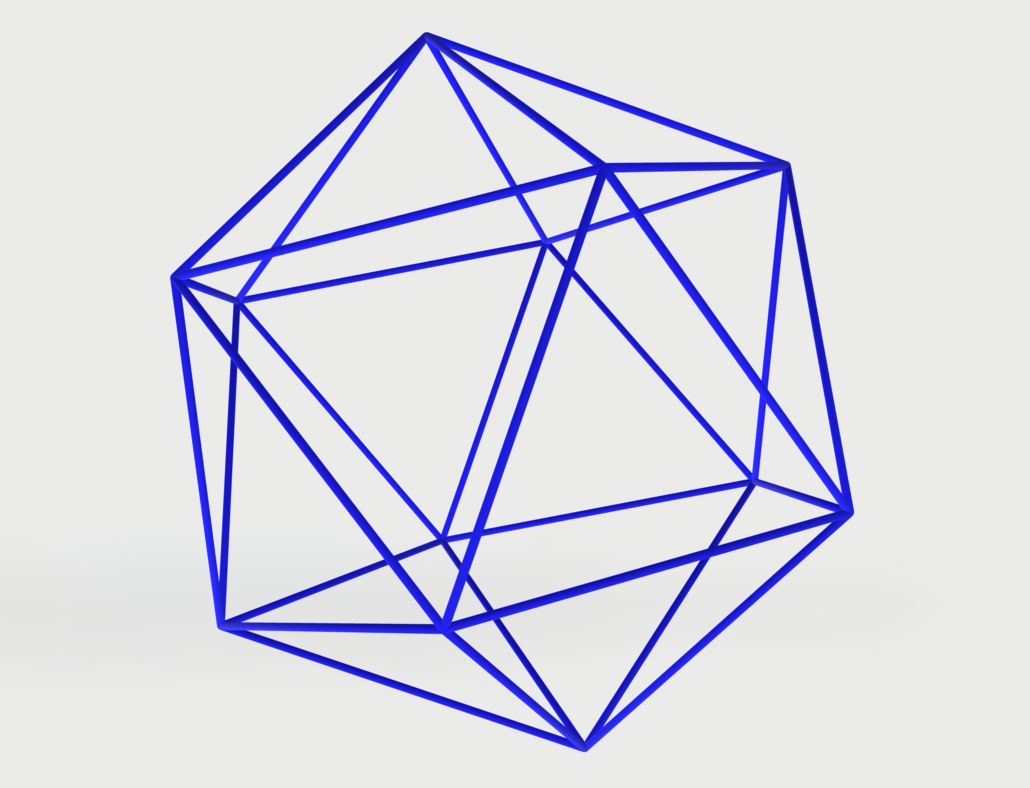}.
\end{center}

\subsection{4-Simplex}
A simple example past three dimensions.
\par\noindent\rule{\textwidth}{0.4pt}
\tiny
\begin{verbatim}
from shape_existence.complexes_and_proofs import AbstractSimplicialComplex, Fraction
ASC = AbstractSimplicialComplex

four_simplex = ASC(mode = "maximal_simplices", data = [["a", "b", "c", "d", "e"]])
four_simplex_realized = four_simplex.heuristic_embed(dim = 4, desired_sq_lengths = {"default" : 1}, final_round_digits = 9)
four_simplex_realized.save_as_obj("four_simplex.obj", "./obj_files/")
four_simplex_realized.prove_existence(desired_sq_lengths = {"default" : Fraction(1)}, verbose = True)
\end{verbatim}
\par\noindent\rule{\textwidth}{0.4pt}
\vspace{.1cm}
\normalsize
And here is the resulting successful proof log:\\
\par\noindent\rule{\textwidth}{0.4pt}
\tiny
\begin{verbatim}
Attempting to prove existence

Starting realization:
	Abstract data:	
		mode: maximal_simplices
		data: [['a', 'b', 'c', 'd', 'e']]
	Coordinate Data:
		c : [7256651 / 10000000, 7642927 / 8000000, 674111317 / 1000000000, 171828441 / 200000000]
		a : [10745679 / 200000000, 636560891 / 1000000000, 339792449 / 1000000000, 280268003 / 1000000000]
		e : [828590217 / 1000000000, 293617321 / 1000000000, 851392509 / 1000000000, 137985737 / 1000000000]
		b : [474175227 / 500000000, 59615879 / 62500000, 3560127 / 31250000, 61270543 / 1000000000]
		d : [395172831 / 500000000, 103949773 / 500000000, 1266793 / 40000000, 175750757 / 250000000]

Desired square lengths:
	default : 1

Checking inequality 1:
	 d  = 4
	|V| = 5
	|E| = 10
	Success: d|V| >= |E|

Checking self-intersection:
	Square collision distance = 1220703126314795045203246989624242191813040793340115540477650010506609 / 2929687508261558621408864856739271611982097732639332277343750000000000
	Collision distance in [32274861 / 50000000, 64549723 / 100000000] ~ [0.6455, 0.6455]
	Success: starting realization non-self-intersecting

Checking inequality 2:
	sigma_min in [9999 / 5000, 2] ~ [1.9998, 2.0]
	Success: sigma_min > 0

Checking inequality 3:
	rho_squared = 499054969996360073 / 100000000000000000000000000000000000
	rho in [111697691 / 50000000000000000, 611697691 / 50000000000000000] ~ [0.0, 0.0]
	sigma_min ^ 2 / (16 * E ^ .5) in [99980001 / 1264911068, 12500000 / 158113883] ~ [0.07904, 0.07906]
	Success: rho < sigma_min ^ 2 / (16 * E ^ .5)

Checking inequality 4:
	LHS NUM := sigma_min - [sigma_min ^ 2 - 16 * rho * |E| ^ .5 ] ^ .5 in [-9999 / 50000000, 1251 / 6250000] ~ [-0.0002, 0.0002]
	LHS DEN := 8 * |E| ^ .5 in [158113883 / 6250000, 316227767 / 12500000] ~ [25.29822, 25.29822]
	LHS     := (LHS NUM) / (LHS DEN) in [-9999 / 1264911064, 1251 / 158113883] ~ [-1e-05, 1e-05]
	CD / |V| ^ .5 in [32274861 / 111803399, 64549723 / 223606797] ~ [0.28868, 0.28868]
	Success: LHS < CD / |V| ^ .5

Success: existence proven
\end{verbatim}
\par\noindent\rule{\textwidth}{0.4pt}
\normalsize
\vspace{.1cm}

The \texttt{save\_to\_obj} function produces a 3D model by using the first 3 coordinates of each point in the initial realization. Here is a screenshot of a produced model:\\
\begin{center}\includegraphics[width = .4\linewidth]{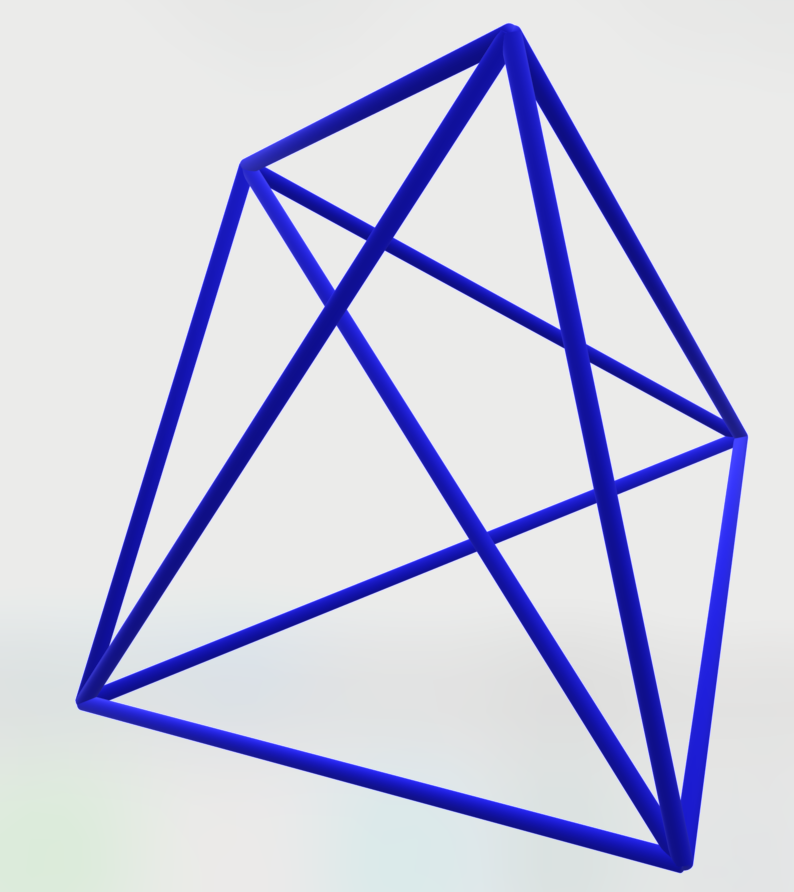}.
\end{center}

\section{A failed proof}\label{nonexamples}
Sometimes our proof technique does not succeed. Consider the hexagonal antiprism pictured below:\\
\begin{center}\includegraphics[width = .4\linewidth]{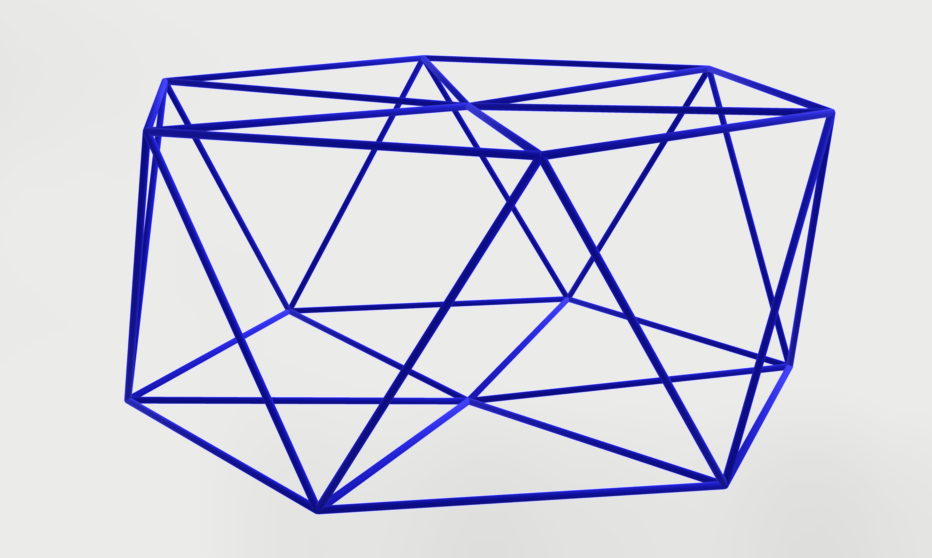}.
\end{center}
Here is sample code input and output:\\
\footnotesize
\par\noindent\rule{\textwidth}{0.4pt}
\begin{verbatim}
from shape_existence.complexes_and_proofs import AbstractSimplicialComplex, Fraction
ASC = AbstractSimplicialComplex

bad_antiprism = ASC(mode = "maximal_simplices", data = [["tm", "t1", "t2"], ["tm", "t2", "t3"], ["tm", "t3", "t4"],
                                                        ["tm", "t4", "t5"], ["tm", "t5", "t6"], ["tm", "t6", "t1"],
                                                        ["bm", "b1", "b2"], ["bm", "b2", "b3"], ["bm", "b3", "b4"],
                                                        ["bm", "b4", "b5"], ["bm", "b5", "b6"], ["bm", "b6", "b1"],
                                                        ["t1", "b1", "b2"], ["t2", "b2", "b3"], ["t3", "b3", "b4"],
                                                        ["t4", "b4", "b5"], ["t5", "b5", "b6"], ["t6", "b6", "b1"],
                                                        ["b1", "t1", "t6"], ["b2", "t2", "t1"], ["b3", "t3", "t2"],
                                                        ["b4", "t4", "t3"], ["b5", "t5", "t4"], ["b6", "t6", "t5"]])
bad_antiprism_realized = bad_antiprism.heuristic_embed(dim = 3, desired_sq_lengths = {"default" : 1}, final_round_digits = 9)
bad_antiprism_realized.save_as_obj("bad_antiprism.obj", "./obj_files/")
bad_antiprism_realized.prove_existence(desired_sq_lengths = {"default" : 1}, verbose = True)
\end{verbatim}
\par\noindent\rule{\textwidth}{0.4pt}
\footnotesize
\begin{verbatim}
Attempting to prove existence

Starting realization:
	Abstract data:	
		mode: maximal_simplices
		data: [['tm', 't1', 't2'], ['tm', 't2', 't3'], ['tm', 't3', 't4'], ['tm', 't4', 't5'], ['tm', 't5', 't6'], 
  ['tm', 't6', 't1'], ['bm', 'b1', 'b2'], ['bm', 'b2', 'b3'], ['bm', 'b3', 'b4'], ['bm', 'b4', 'b5'],
  ['bm', 'b5', 'b6'], ['bm', 'b6', 'b1'],['t1', 'b1', 'b2'], ['t2', 'b2', 'b3'], ['t3', 'b3', 'b4'],
  ['t4', 'b4', 'b5'], ['t5', 'b5', 'b6'], ['t6', 'b6', 'b1'],['b1', 't1', 't6'], ['b2', 't2', 't1'], 
  ['b3', 't3', 't2'], ['b4', 't4', 't3'], ['b5', 't5', 't4'], ['b6', 't6', 't5']]
	Coordinate Data:
		bm : [242420139 / 1000000000, 88209211 / 200000000, 347417067 / 500000000]
		t2 : [64955007 / 62500000, 1460273147 / 1000000000, 308853307 / 500000000]
		b3 : [40363343 / 1000000000, 283159893 / 200000000, 37309113 / 62500000]
		b4 : [-35285911 / 125000000, 766903711 / 1000000000, -45931877 / 500000000]
		b6 : [391400809 / 1000000000, -269887657 / 500000000, 823167127 / 1000000000]
		b1 : [357116567 / 500000000, 109120773 / 1000000000, 302372923 / 200000000]
		b5 : [-106536099 / 1000000000, -105372201 / 500000000, 21017483 / 1000000000]
		t1 : [1427998777 / 1000000000, 141202073 / 200000000, 1145333333 / 1000000000]
		b2 : [21551239 / 40000000, 543390033 / 500000000, 55951509 / 40000000]
		t4 : [240379443 / 500000000, 337122819 / 1000000000, -574688421 / 1000000000]
		t3 : [70664569 / 125000000, 1275577083 / 1000000000, -242530763 / 1000000000]
		t6 : [268599499 / 200000000, -116264219 / 500000000, 813382133 / 1000000000]
		t5 : [869599511 / 1000000000, -26056471 / 62500000, -727601 / 15625000]
		tm : [20084469 / 20000000, 527264939 / 1000000000, 256555313 / 1000000000]

Desired square lengths:
	default : 1

Checking inequality 1:
	 d  = 3
	|V| = 14
	|E| = 36
	Success: d|V| >= |E|

Checking self-intersection:
	Square collision distance = [LONG FRACTION OMITTED]
	Collision distance in [79145971 / 100000000, 19786493 / 25000000] ~ [0.79146, 0.79146]
	Success: starting realization non-self-intersecting

Checking inequality 2:
	sigma_min in [549 / 5000, 11 / 100] ~ [0.1098, 0.11]
	Success: sigma_min > 0

Checking inequality 3:
	rho_squared = 19213882758715484274732620078053 / 500000000000000000000000000000000000
	rho in [619901 / 100000000, 309951 / 50000000] ~ [0.0062, 0.0062]
	sigma_min ^ 2 / (16 * E ^ .5) in [301401 / 2400000004, 121 / 960000] ~ [0.00013, 0.00013]
	Failed: unable to verify rho < sigma_min ^ 2 / (16 * E ^ .5)
\end{verbatim}
\par\noindent\rule{\textwidth}{0.4pt}
\vspace{.1cm}
\normalsize
We see the proof fails because the 3rd inequality is not satisfied. In general, this happens due to some combination of two factors --- (1) the initial lengths are insufficiently close to the desired lengths and (2) the lowest singular value is not sufficiently large --- and so we can not guarantee a nearby realization with the desired lengths. In this case the failure is due to $\sigma_{\min}$, which is zero at a realization with the desired edge lengths\footnote{To see this, first note that the Jacobian is a map from $\mathbb{R}^{42}$ to $\mathbb{R}^{36}$. We lose 6 rank from infinitesimal translations and rotations of the coordinates, and 2 more rank from the top and bottom vertices in the middle of the hexagons, which, when moved perpendicular to the hexagon through their neighbors, only changes edge lengths to second order. Thus there can be at most 34 non-zero singular values.}. Because of this, we believe Theorem \ref{main} cannot be applied to prove existence (unless the starting realization already has the desired lengths).

\section{Final Notes}
 We have used this technique to examine triangulations of the sphere with at most 6 triangles meeting at a given vertex. We have used the above-described technique to prove thousands of these triangulations can be embedded in $\mathbb{R}^3$ with unit-length edges. In fact, we have shown every such triangulation with up to 23 vertices exists.  Check out the online gallery here: \url{math.dartmouth.edu/~mellison/tripolys}.

\section{Acknowledgements}
Thank you to Zili Wang and Peter Doyle for many helpful discussions!
\printbibliography

\appendix
\section{Rational bounds on the lowest singular value
of a matrix with rational entries}\label{bound sigma}

Let $A$ be a matrix with rational entries and smallest singular value $\sigma_{\min}$. We give a procedure to give a good rational interval containing $\sigma_{\min}$.
Note that the procedure yields correct bounds if it runs to completion, but may (theoretically) encounter an error.

\begin{enumerate}
    \item Using any standard scientific library, compute a floating point approximation $f$ of the smallest singular values of $A$ (converting the fractions to floating point first).
    \item Round $f$ down and up a few  decimal points to obtain rational numbers $\sigma_{lb}, \sigma_{ub}$. Set $\sigma_{lb}$ to 0 if negative.
    \item Let $B=A^TA$ or $AA^T$, whichever has smaller dimensions (and choosing, say, $A^TA$ if it's a tie).
    \item Prove $\sigma_{lb}$ is indeed a lower bound on $\sigma_{\min}$ by proving $$B - \sigma_{lb}^2 I$$ is positive definite, and prove $\sigma_{ub}$ is an upper bound on $\sigma_{\min}$ by proving 
    $$B - \sigma_{ub}^2 I$$ is not positive definite.\footnote{We can decide whether a symmetric matrix of rationals is positive definite by applying Sylvester's Criterion. \cite{horn}}
\end{enumerate}
If the positive-definiteness checks succeed, and in practice they do, we now have $\sigma_{\min} \in [\sigma_{lb}, \sigma_{ub}]$.

\section{Rational bounds on the square root of a
rational number}\label{bound sqrt}
Let $x$ be a rational number for which we wish to obtain good rationally bounds on $\sqrt{x}$. We give a simple procedure to obtain these. Once again, note that the procedure yields correct bounds if it runs to completion, but may (theoretically) encounter an error.
\begin{enumerate}
    \item Convert the rational number $x$ to a floating point number and obtain a floating point square root $f$. 
    \item Round $f$ up and down a few decimal points to obtain rational numbers $l,u$. Set $l$ to 0 if it's negative.
    \item Verify $l^2 \leq x \leq u^2.$
\end{enumerate}
If these steps all succeed, we now have proven rational bounds on the square root of a rational number.\\

One can apply a similar idea to rationally bound the square root on a rational interval --- just lower bound the square root of the lower endpoint and upper bound the square root of the upper. This allows square roots to be incorporated into a rational interval arithmetic.

\section{Exactly computing the square distance between two convex sets}\label{convex distance}
Let $X$, $Y$ be convex sets in $\mathbb{R}^d$ with respective vertices $x_1, ..., x_m$ and $y_1, ..., y_n$. Suppose the vertices have rational coordinates.\footnote{This all works over any ordered field.} \\

Determining the shortest squared distance between the $X$ and $Y$ is a quadratic programming problem:
\begin{align*}
    \text{minimize }& ||\sum_{i=1}^m{\alpha_i}x_i - \sum_{i=1}^n{\beta_i}x_i||^2\\
    \text{such that }&\sum \alpha_i = \sum \beta_i = 1; \alpha_i, \beta_i \geq 0.
\end{align*}
Note here that $\alpha_i, \beta_i$ are barycentric coordinates. \\

We can bring this into the standard form of a quadratic program as follow. Let $P$ be the matrix with columns $x_i$ and let $Q$ be the matrix with columns $y_i$. Let $\alpha, \beta$ be the respective column vectors of the $\alpha_i, \beta_i$. Let $0_k, 1_k$ denote row vectors of zeroes and ones of length $k$.\\
Then our problem becomes
\begin{align*}
    \text{minimize }& 
    \begin{pmatrix}
        \alpha^T& \beta^T
    \end{pmatrix}
    \begin{pmatrix}
        X^T X & -X^T Y\\
        - Y^T X & Y^T Y
    \end{pmatrix}
    \begin{pmatrix}
        \alpha \\
        \beta
    \end{pmatrix}\\
    \text{such that }&
    \begin{pmatrix}
        1_m & 0_n\\
        -1_m & 0_n\\
        0_m & 1_n\\
        0_m & -1_n
    \end{pmatrix}
    \begin{pmatrix}
        \alpha \\
        \beta
    \end{pmatrix}
    \geq \begin{pmatrix}
        1\\
        -1\\
        1\\
        -1
    \end{pmatrix},\ 
    \begin{pmatrix}
        \alpha \\
        \beta
    \end{pmatrix}
    \geq 0.
\end{align*}
The matrix 
$$\begin{pmatrix}
        X^T X & -X^T Y\\
        - Y^T X & Y^T Y
\end{pmatrix}$$
is positive semi-definite, and so this is a convex quadratic program. One method to solve it exactly is to first convert it to a linear complementarity problem (LCP) and then solve this LCP with Lemke's algorithm. \cite{lemke}\cite{eberly} Note that this solution process will not leave the field generated by the coefficients, in our case the rationals.\\

The above procedure will exactly compute the squared distance between two convex bodies with rational coordinates, and, if desired, we may then apply our rational square root procedure to obtain a good rational interval containing this distance.\\

Note that while finding the minimum square distance is a quadratic program, deciding if the bodies intersect is, using the well-known reductions, a linear program (see e.g. \cite{vaserstein}):\\
\begin{align*}
    \text{determinine if there exist }& \alpha_i, \beta_i \geq 0\\
    \text{such that }& \sum_{i=1}^m{\alpha_i}x_i = \sum_{i=1}^n{\beta_i}x_i\\
    \text{and }&\sum \alpha_i = \sum \beta_i = 1.\\
\end{align*}

\subsection{Example distance computation}
Here is an example where this computation is performed using the \texttt{shape-existence} library to compute the distance between a line segment and triangle in $\mathbb{E}^3$:\\
\par\noindent\rule{\textwidth}{0.4pt}
\begin{verbatim}
from shape_existence.flexible_cqp import simplex_square_distance, Fraction

triangle = [[3,0,0],[0,3,0],[0,0,3]]
segment = [[0,1,1],[1,0,1]]
print(simplex_square_distance(triangle, segment, map_to_type = Fraction))
\end{verbatim}
\par\noindent\rule{\textwidth}{0.4pt}
\vspace{.1cm}
Here is the output:\\
\par\noindent\rule{\textwidth}{0.4pt}
\begin{verbatim}
(1 / 3, ([1 / 3, 4 / 3, 4 / 3], [0, 1, 1]))
\end{verbatim}
\par\noindent\rule{\textwidth}{0.4pt}
\vspace{.1cm}
Parsing the output, the shortest square distance is $1/3$, and the pair of closest points is $(1/3,\ 4/3,\ 4/3)$ and $(0,\ 1,\ 1)$ on the triangle and segment respectively.
\end{document}